\documentclass[12pt]{amsart}

\usepackage{graphics, stackrel}
\usepackage{amsmath, amssymb, amsthm}
\usepackage{graphicx}
\usepackage{verbatim}
\usepackage{amsfonts}
\usepackage[numbers]{natbib}
\usepackage{fancyvrb}
\usepackage{color}
\usepackage{mdwlist}

\usepackage{enumitem}
\setlist[enumerate]{itemsep=2pt,topsep=3pt}
\setlist[itemize]{itemsep=2pt,topsep=3pt}
\setlist[enumerate,1]{label=(\alph*)}

\usepackage[citecolor=blue, colorlinks=true, linkcolor=blue]{hyperref}

\usepackage{mathrsfs}
\usepackage{bbm}

\usepackage[left=1.25in, right=1.25in, top=1.0in, bottom=1.25in, includehead, includefoot]{geometry}

\renewcommand{\leq}{\leqslant}
\renewcommand{\geq}{\geqslant}

\usepackage{mathtools}

\DeclarePairedDelimiter\floor{\lfloor}{\rfloor}

\newcommand{\setntn}[2]{ \{ #1 : #2 \} }

\newcommand{\preqsd}{\preceq_{sd} }

\newcommand{\toweak}{\stackrel { w } {\to} }

\newcommand{\disteq}{\stackrel { \mathscr D } {=} }

\newcommand{\1}{\mathbbm 1}

\newcommand{\A}{\forall}
\newcommand*\diff{\mathop{}\!\mathrm{d}}

\newcommand{\cC}{\mathcal C}
\newcommand{\ccC}{\mathscr C}

\newcommand{\bB}{\mathcal B}
\newcommand{\oO}{\mathcal O}

\newcommand{\mM}{\mathscr M}

\newcommand{\fF}{\mathscr F}

\newcommand{\pP}{\mathscr P}

\newcommand{\RR}{\mathbbm R}

\newcommand{\NN}{\mathbbm N}
\newcommand{\PP}{\mathbbm P}
\newcommand{\GG}{\mathbbm G}
\newcommand{\EE}{\mathbbm E \,}

\newcommand{\bQ}{\mathbf Q}

\theoremstyle{plain}

\newtheorem{theorem}{Theorem}[section]
\newtheorem{corollary}{Corollary}[section]
\newtheorem{lemma}{Lemma}[section]
\newtheorem{proposition}{Proposition}[section]

\theoremstyle{definition}

\newtheorem{example}{Example}[section]
\newtheorem{remark}{Remark}[section]

\newtheorem{assumption}{Assumption}[section]

\begin{document}

\title[Classical and Monotone Markov Chains]{A Unified Stability Theory for Classical and Monotone Markov Chains}

\author{Takashi Kamihigashi}
\address{RIEB, Kobe University}
\email{tkamihig@rieb.kobe-u.ac.jp}

\author{John Stachurski}
\address{Research School of Economics, Australian National University}
\email{john.stachurski@anu.edu.au}
\thanks{This paper was written in part while the second author was visiting RIEB at Kobe
      University.  We have benefited from financial support from the Japan
      Society for the Promotion of Science (KAKENHI 15H05729) and Australian
      Research Council Discovery Grant DP120100321.}

\date{\today}

\begin{abstract}
    This paper integrates two strands of the literature on stability of
    general state Markov chains: conventional, total variation based results
    and more recent order-theoretic results.  First we introduce a complete
    metric over Borel probability measures based on ``partial'' stochastic
    dominance.  We then show that many conventional results framed in the
    setting of total variation distance have natural generalizations to the
    partially ordered setting when this metric is adopted.

    \vspace{1em}


\end{abstract}

\maketitle

\section{Introduction}

\label{s:intro}

Following the work of Wolfgang Doeblin \cite{doeblin1938expose,
doeblin1938proprietes, doeblin1940elements}, many classical results from Markov
chain theory have built on fundamental connections between total
variation distance, Markov chains and couplings.  For some models, however,
total variation convergence is too strong.  In
response, researchers have developed an alternative methodology based on monotonicity 
\cite{dubins1966invariant, yahav1975fixed, bhattacharya1988asymptotics}.    In this line of research, transition probabilities are
assumed to have a form of monotonicity not required in the classical theory.
At the same time, mixing conditions are generally weaker, as is the
notion of convergence to the stationary distribution.\footnote{Further contributions
to this approach can be found in \cite{razin1979stochastic,
hopenhayn1992stochastic, bhattacharya2010limit, kamihigashi2012order}.
For some recent extensions and applications in economics see
\cite{kamihigashi2014stochastic}.}

To give one example, consider a Markov chain $\{X_t\}$ defined by
\begin{equation}
    \label{eq:bern}
    X_{t+1} = \frac{X_t + W_{t+1}}{2}
\end{equation}
where $\{W_t\}_{t \geq 1}$ is an {\sc IID} Bernoulli$(1/2)$ random sequence, taking
values $0$ and $1$ with equal probability.  For the state space take
$S = [0, 1]$.  Let $P^t(x, \cdot)$ be the distribution of $X_t$ given $X_0 =
x \in S$.  Clearly, if $X_t$ is a rational number in $S$, then so is
$X_{t+1}$.  Similarly, if $X_t$ is irrational, then so is $X_{t+1}$.
Thus, if $x$ and $y$ are rational and irrational respectively, the
distributions $P^t(x, \cdot)$ and $P^t(y, \cdot)$ are concentrated on disjoint
sets,  and hence,  when $\| \cdot \|$ is the total variation norm,
\begin{equation}
    \label{eq:sm}
    \| P^t(x, \cdot) - P^t(y, \cdot) \| = 2
\end{equation}
for all $t \in \NN$.  Total variation convergence fails for this class of
models.

At the same time, the right hand side of \eqref{eq:bern} is increasing in
the current state for each fixed value of the shock $W_{t+1}$.  Moreover,
trajectories mix in a monotone sense: A trajectory starting at $X_0 = 0$ can
approach $1$ with a suitable string of shocks and a trajectory starting at $1$
can approach $0$.  Using these facts one can show using the results in
\cite{bhattacharya1988asymptotics}, say, that a unique stationary distribution
exists and the distribution of $X_t$ converges to it in a complete metric
defined over the Borel probability measures that is weaker than total
variation convergence.

This is one example where monotone methods can be used to establish some form
of stability, despite the fact that the classical conditions based around
total variation convergence fail. Conversely, there are many models that the
monotone methods developed in \cite{bhattacharya1988asymptotics} and related
papers cannot accommodate, while the classical theory based around total
variation convergence handles them easily.  One example is the simple
``inventory'' model 
\begin{equation}
    \label{eq:invent}
    X_{t+1} = 
    \begin{cases}
        (X_t - W_{t+1})_+ & \quad\text{if } X_t > 0
        \\
        (K - W_{t+1})_+ & \quad\text{if } X_t = 0,
    \end{cases}
\end{equation}
where $x_+ := \max\{x, 0\}$.  Again $\{W_t\}$ is {\sc IID}. Assume
that $\ln W_t$ is standard normal.  The state space we take to be $S = [0, K]$.  Figure~\ref{f:inventory} shows a
typical trajectory when $K=100$ and $X_0 = 50$.  

\begin{figure}
    \centering
    \scalebox{0.6}{\includegraphics[clip=true, trim=0mm 0mm 0mm 0mm]{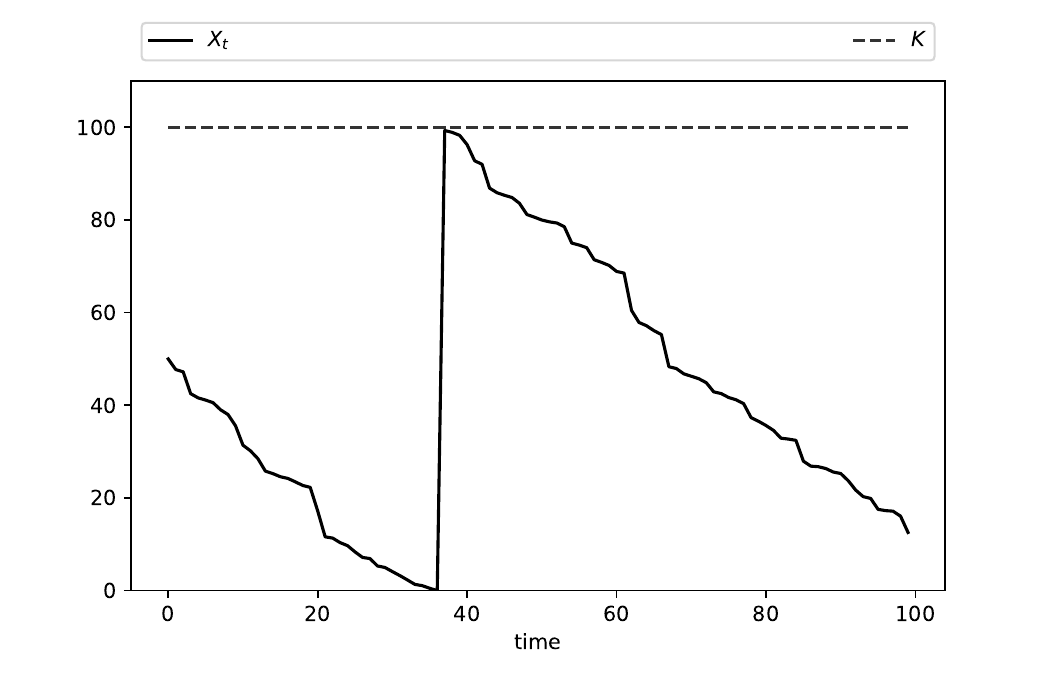}}
    \caption{\label{f:inventory} A time series from the inventory model}
\end{figure}

On one hand, the monotone methods in \cite{bhattacharya1988asymptotics} cannot be applied here
because of a failure of monotonicity with respect to the standard ordering of
$\RR$.  On the other hand, the classical theory based around total variation
convergence is straightforward to apply.  For example, one can use Doeblin's
condition (see, e.g., \cite{meyn2012markov}, theorem~16.2.3) to show the
existence of a unique stationary distribution to which the distribution of
$X_t$ converges in total variation, regardless of the distribution of $X_0$. 
In the terminology of \cite{meyn2012markov}, the process is uniformly ergodic.

The purpose of this paper is to show that both of these stability results
(i.e, the two sets of results concerning the two models \eqref{eq:bern} and
\eqref{eq:invent}), which were based on two hitherto separate approaches, can be derived from the same
theoretical framework.  More generally, we construct stability results that
encompasses all uniformly ergodic models in the sense of \cite{meyn2012markov}
and all monotone models shown to be stable in
\cite{bhattacharya1988asymptotics}, as well as extending to other monotone or
partially monotone models on state spaces other than $\RR^n$.

We begin our analysis by introducing what is shown to be a
complete metric $\gamma$ on the set of Borel probability measures on a
partially ordered Polish space that includes total variation distance, the
Kolmogorov uniform distance and the Bhattacharya distance
\cite{bhattacharya1988asymptotics, chakraborty1998completeness} as special
cases.  We show that many fundamental concepts from conventional Markov chain
theory using total variation distance and coupling have direct generalizations
to the partially ordered setting when this new metric is adopted.  Then, by
varying the choice of partial order, we recover key aspects of both classical
total variation based stability theory and monotone methods as special
cases.\footnote{There is one additional line of research that deals with Markov
    models for which the classical conditions of irreducibility and total
    variation convergence fail.  In this line of analysis, irreducibility is
    replaced by an assumption that the law of motion for the state is itself
    contracting ``on average,'' and this contractivity is then passed on to
    an underlying metric over distributions that conforms in some way to the
    topology of the state space.  See, for example,
\cite{diaconis1999iterated} or \cite{wu2004limit}.  Such results can
be used to show stability of our first example, which contracts on average
with respect to the usual metric on $\RR$.  On the other hand, it cannot be
applied to our second (i.e., inventory) example using the same metric, since the law of motion
contains a jump.  In \cite{bhattacharya1988asymptotics} and
\cite{hopenhayn1992stochastic} one can find many other applications where
monotone methods can be used---including the results developed here---while the ``contraction on average'' conditions
of \cite{diaconis1999iterated} and \cite{wu2004limit} do not hold.   In
general, our results should be understood as complements rather than
substitutes when compared to average contractions.}

After preliminaries, we begin with a discussion of ``ordered'' affinity, which
generalizes the usual notion of affinity for measures.  The concept of ordered
affinity is then used to define the total ordered variation metric. Throughout
the paper, longer proofs are deferred to the appendix.  The conclusion
contains suggestions for future work.

\section{Preliminaries}

\label{s:prel}

Let $S$ be a Polish (i.e., separable and completely metrizable) space, let $\oO$ be the open sets, let $\cC$ be the closed
sets and let $\bB$ be the
Borel sets. Let  $\mM_s$ denote the set of all finite signed measures on $(S,
\bB)$.  In other words, $\mM_s$ is all countably additive set functions from
$\bB$ to $\RR$.  Let $\mM$ and $\pP$ be the finite measures and probability
measures in $\mM_s$, respectively.  
If $\kappa$ and $\lambda$ are
in $\mM_s$, then $\kappa \leq \lambda$ means that $\kappa(B) \leq \lambda(B)$
for all $B \in \bB$.    
The symbol $\delta_x$ denotes the probability measure concentrated on $x \in
S$.

Let $bS$ be the set of all bounded $\bB$-measurable functions from $S$ into
$\RR$.  If $h \in bS$ and $\lambda \in \mM_s$, then $\lambda(h) := \int h \diff \lambda$.
For $f$ and $g$ in $bS$, the statement $f \leq g$ means that $f(x) \leq g(x)$ for all $x \in S$.
Let
\begin{equation*}
    H := \setntn{ h \in bS }{ -1 \leq h \leq 1}
    \quad \text{and} \quad
    H_0 := \setntn{ h \in bS }{ 0 \leq h \leq 1}.
\end{equation*}
The \emph{total variation norm} of $\lambda \in \mM_s$ is $\| \lambda \| := \sup_{h \in H} |\lambda(h)|$.
Given $\mu$ and $\nu$ in $\pP$, a random element $(X,Y)$ taking values in $S
\times S$ and defined on a common probability space $(\Omega, \fF, \PP)$ is called a
\emph{coupling} of $(\mu, \nu)$ if $\mu = \PP \circ X^{-1}$ and $\nu = \PP
\circ Y^{-1}$ (i.e., if the distribution of $(X,Y)$ has marginals $\mu$ and $\nu$ respectively---see, e.g.,
\cite{lindvall2002lectures} or \cite{thorisson2000coupling}). The set of all
couplings of $(\mu, \nu)$ is denoted below by $\ccC(\mu, \nu)$.  
A sequence $\{\mu_n\} \subset \pP$ converges to $\mu \in \pP$
\emph{weakly} if $\mu_n(h) \to \mu(h)$ as $n \to \infty$ for all
continuous $h \in bS$.  In this case we write  $\mu_n \toweak \mu$.

Given $\mu$ and $\nu \in \mM$, their measure theoretic \emph{infimum}
$\mu \wedge \nu$ is the largest element of $\mM$ dominated by both $\mu$ and
$\nu$.  It can be defined by taking $f$ and $g$ to be densities of $\mu$ and
$\nu$ respectively under the dominating measure $\lambda := \mu + \nu$ and defining $\mu
\wedge \nu$ by $(\mu \wedge \nu)(B) := \int_B \, \min\{ f(x), g(x) \} \lambda(\diff x)$ for all $B \in \bB$.
The total variation distance 
between $\mu$ and $\nu$ is related to $\mu \wedge \nu$ via
    $\| \mu - \nu \| 
    = \| \mu \| + \| \nu \| - 2 \| \mu \wedge \nu \|$.
See, for example, \cite{pollard2002user}.  For probability measures we also have 
\begin{equation}
    \label{eq:sisf}
    \sup_{B \in \bB} \{ \mu(B) - \nu(B) \}
    = \sup_{B \in \bB} |\mu(B) - \nu(B) |
    = \| \mu - \nu \| / 2.
\end{equation}

The \emph{affinity} between two measures $\mu, \nu$ in $\mM$ is the value
    $\alpha(\mu, \nu) := (\mu \wedge \nu)(S)$.
The following properties are elementary:

\begin{lemma}
    \label{l:poaf}
    For all $(\mu, \nu) \in \mM \times \mM$ we have
    \begin{enumerate}
        \item $0 \leq \alpha(\mu, \nu) \leq \min\{\mu (S), \nu(S) \}$
        \item $\alpha(\mu, \nu) = \mu(S) = \nu(S)$ if and only if $\mu = \nu$.
        \item $\alpha(c \mu, c \nu) = c\alpha(\mu, \nu)$ for all $c \geq 0$.
    \end{enumerate}
\end{lemma}

There are several other common representations of affinity.  
For example, when $\mu$ and $\nu$ are both probability measures, we have
\begin{equation}
    \label{eq:affc}
    \alpha(\mu, \nu) 
    = 1 - \sup_{B \in \bB} | \mu(B) - \nu(B) |
    = \max_{(X, Y) \in \ccC(\mu, \nu)} \PP\{ X = Y \}.
\end{equation}
(See, e.g., \cite{pollard2002user, lindvall2002lectures}.) The second equality
in \eqref{eq:affc} states that, if $(X,Y) \in \ccC(\mu, \nu)$, then $\PP\{X =
Y\} \leq \alpha(\mu, \nu)$, and, moreover, there exists a $(X, Y) \in
\ccC(\mu, \nu)$ such that equality is attained.  Any such coupling is called a
\emph{maximal} or \emph{gamma} coupling.  See theorem 5.2 of
\cite{lindvall2002lectures}.  From \eqref{eq:sisf} and \eqref{eq:affc} we obtain
\begin{equation}
    \label{eq:tvmtap}
    \| \mu - \nu \| = 2(1 - \alpha(\mu, \nu)).
\end{equation}

\section{Ordered Affinity}

\label{s:otg}

We next introduce a generalization of affinity when $S$ has a partial order.
We investigate its properties in detail, since both our metric and 
the stability theory presented below rely on this concept.

\subsection{Preliminaries}

As before, let $S$ be a Polish space.  A closed partial order $\preceq$ on
$S$ is a partial order $\preceq$ such that its graph 
\begin{equation*}
    \GG := \setntn{(x, y) \in S \times S}{x \preceq y}
\end{equation*}
is closed in the product topology.  In the sequel, a \emph{partially ordered
Polish space} is any such pair $(S, \preceq)$, where $S$ is nonempty
and Polish, and $\preceq$ is a closed partial order on $S$.  When no confusion
arises, we denote it simply by $S$.\footnote{The partial order is assumed to
    be closed in the theory developed below because we build a metric over
    Borel probability measures that
    depends on this partial order and closedness of the order is required to show
that the metric is complete.}

For such a space $S$, we call $I \subset S$ \emph{increasing} if $x \in I$ and $x \preceq y$ implies $y \in I$.    We call
$h \colon S \to \RR$ \emph{increasing} if $x \preceq y$ implies $h(x) \leq
h(y)$.  We let $i\bB$, $i\oO$ and $i\cC$ denote the increasing Borel, open and
closed sets, respectively, while $ibS$ is the increasing functions in $bS$.  In
addition, 
\begin{itemize}
    \item $iH := H \cap ibS = \setntn{h \in ibS}{-1 \leq h \leq 1}$ and
    \item $iH_0 := H_0 \cap ibS = \setntn{h \in ibS}{0 \leq h \leq 1}$.
\end{itemize}
If $B \in \bB$, then
    $i(B)$ is all $y \in S$ such that $x \preceq y$ for some $x \in B$,
    while $d(B)$ is all $y \in S$ such that $y \preceq x$ for some $x \in B$.
Given $\mu$ and $\nu$ in $\mM$, we say that $\mu$ is
\emph{stochastically dominated} by $\nu$ and write $\mu \preqsd \nu$ if
$\mu(S) = \nu(S)$ and $\mu(I) \leq \nu(I)$ for all $I \in i\bB$.
Equivalently,
$\mu(S) = \nu(S)$ and $\mu(h) \leq \nu(h)$ for all $h$ in $iH$ or
$iH_0$.  See \cite{kamae1978stochastic}.

One important partial order on $S$ is the \emph{identity order}, where $x
\preceq y$ if and only if $x = y$.  Then $i\bB = \bB$, $ibS = bS$, $iH = H$,
$iH_0 = H_0$ and $\mu \preqsd \nu$ if and only if $\mu = \nu$.

\begin{remark}
    \label{r:sfe}
    Since $S$ is a partially ordered Polish space, for any $\mu,\nu$
    in $\pP$ we have $\mu = \nu$ whenever $\mu(C) = \nu(C)$ for all $C \in
    i\cC$, or, equivalently, $\mu(h) = \nu(h)$ for all continuous $h \in ibS$.
    See \cite[lemma~1]{kamae1978stochastic}.  Hence
    $\mu \preqsd \nu$ and $\nu \preqsd \mu$ imply $\mu = \nu$.
\end{remark}

\begin{lemma}
    \label{l:cfs}
    If $\lambda \in \mM_s$, then
        $\sup_{I \in i\bB} \lambda(I) = \sup_{h \in iH_0} \lambda(h)$
    and
    \begin{equation}
        \label{eq:cfs2}
        \sup_{h \in iH} |\lambda(h)|
        = \max \left\{
            \sup_{h \in iH} \lambda(h),
            \; \sup_{h \in iH} (-\lambda)(h)
            \right\}.
    \end{equation}
\end{lemma}

The proof is in the appendix.  One can easily check that
\begin{equation}
    \label{eq:chho}
    \lambda \in \mM_s
    \text{ and }  \lambda(S) = 0
    \quad \implies \quad
    \sup_{h \in iH}
    \lambda(h) = 2 \sup_{h \in iH_0} \lambda(h).
\end{equation}

\subsection{Definition of Ordered Affinity}

For each pair $(\mu, \nu) \in \mM \times \mM$, let
\begin{equation*}
    \Phi(\mu, \nu) 
    := \setntn{ (\mu', \nu') \in \mM \times \mM }
        {\mu' \leq \mu,\; \nu' \leq \nu,\; \mu' \preqsd \nu'}.
\end{equation*}
We call $\Phi(\mu, \nu)$ the set of \emph{ordered component pairs} for
$(\mu, \nu)$. Here ``ordered'' means ordered by stochastic dominance.  
The set of ordered component pairs is always nonempty, since
$(\mu \wedge \nu, \mu \wedge \nu)$ is an element of $\Phi(\mu, \nu)$.

\begin{example}
    If $\mu$ and $\nu$ are two measures satisfying $\mu \preqsd \nu$, then
    $(\mu, \nu) \in \Phi(\mu, \nu)$.
\end{example}

\begin{example}
    Given Bernoulli distributions $\mu = (\delta_1 + \delta_2)/2$
    and $\nu = (\delta_0 + \delta_1)/2$, we have  $(\mu', \nu') \in \Phi(\mu,
    \nu)$ when $\mu' = \nu' = \delta_1 / 2$.
\end{example}

We call an ordered component pair $(\mu', \nu') \in 
\Phi(\mu, \nu)$ a \emph{maximal ordered component pair} if
it has greater mass than all others; that is, if
\begin{equation*}
    \mu''(S) \leq \mu'(S) 
    \quad \text{for all } \;
    (\mu'', \nu'') \in \Phi(\mu, \nu).
\end{equation*}
(We can restate this by replacing $\mu'(S)$ and $\mu''(S)$ with $\nu'(S)$ and
$\nu''(S)$ respectively, since
the mass of ordered component pairs is equal by the definition of stochastic
dominance.)  We let $\Phi^*(\mu, \nu)$ denote the set of maximal ordered
component pairs for $(\mu, \nu)$.  Thus, if
\begin{equation}
    \label{eq:dsig}
    \alpha_O(\mu, \nu) 
    := \sup \setntn{ \mu'(S) }{ (\mu', \nu') \in \Phi(\mu, \nu) }.
\end{equation}
then
\begin{equation*}
    \Phi^*(\mu, \nu) 
    = \setntn{(\mu', \nu') \in \Phi(\mu, \nu)}{\mu'(S) = \alpha_O(\mu, \nu)}.
\end{equation*}

Using the Polish space assumption, one can show that maximal ordered component
pairs always exist:

\begin{proposition}
    \label{p:em}
    The set $\Phi^*(\mu, \nu)$ is nonempty for all  $(\mu, \nu) \in \mM \times
    \mM$.
\end{proposition}

\begin{proof}
    Fix $(\mu, \nu) \in \mM \times \mM$ and let $s := \alpha_O(\mu, \nu)$.
    From the definition, we can take sequences
    $\{\mu'_n\}$ and $\{\nu'_n\}$ in $\mM$ such that $(\mu'_n, \nu'_n) \in
    \Phi(\mu, \nu)$ for all $n \in \NN$ and $\mu'_n(S) \uparrow s$.  
    Since $\mu'_n \leq \mu$ and $\nu'_n \leq \nu$ for all $n
    \in \NN$, Prohorov's theorem \cite[theorem~11.5.4]{dudley2002real} implies
    that these sequences have convergent subsequences with $\mu'_{n_k} \toweak \mu'$
    and $\nu'_{n_k} \toweak \nu'$ for some
    $\mu', \nu' \in \mM$.  We claim that $(\mu', \nu')$ is a maximal ordered
    component pair.

    Since $\mu'_{n_k} \toweak \mu'$ and $\mu'_n \leq \mu$ for all $n \in \NN$, 
    theorem 1.5.5 of \cite{lerma2003markov} implies that 
    for any Borel set $B$, we have $\mu'_{n_k}(B) \to \mu'(B)$ in $\RR$ .  
    Hence $\mu'(B) \leq \mu(B)$ and, in particular, $\mu' \leq \mu$.
    An analogous argument gives $\nu' \leq \nu$.
    Moreover, the definition of $\Phi(\mu, \nu)$ and stochastic dominance
    imply that $\mu'_n(S) = \nu'_n(S)$ for all $n \in \NN$, and therefore
    $\mu'(S) = \nu'(S)$.   Also, for any $I \in i\bB$, the fact that
    $\mu'_n(I) \leq \nu'_n(I)$ for all $n \in \NN$ gives us $\mu'(I) \leq
    \nu'(I)$.  Thus, $\mu' \preqsd \nu'$.  Finally, $\mu'(S) = s$, since
    $\mu'_n(S) \uparrow s$.  Hence $(\mu', \nu')$ lies in $\Phi^*(\mu, \nu)$.
\end{proof}

The value $\alpha_O(\mu, \nu)$ defined in \eqref{eq:dsig} gives the mass of
the maximal ordered component pair.  We call it the \emph{ordered affinity} from $\mu$ to $\nu$.
On an intuitive level, we can think of $\alpha_O(\mu, \nu)$ as the ``degree'' to
which $\mu$ is dominated by $\nu$ in the sense of stochastic dominance.
Since $(\mu \wedge \nu, \mu \wedge \nu)$ is an ordered component pair
for $(\mu, \nu)$, we have
\begin{equation}
    \label{eq:aldbal}
    0 \leq \alpha(\mu, \nu) \leq \alpha_O(\mu, \nu),
\end{equation}
where $\alpha(\mu, \nu)$ is the standard affinity
defined in section~\ref{s:prel}.  In fact $\alpha_O(\mu, \nu)$ generalizes the
standard the notion of affinity by extending it to arbitrary partial orders,
as shown in the next lemma.

\begin{lemma}
    If $\preceq$ is the identity order, then $\alpha_O = \alpha$ on $\mM \times \mM$.
\end{lemma}

\begin{proof}
    Fix $(\mu, \nu) \in \mM \times \mM$ and let $\preceq$ be the identity
    order ($x \preceq y$ iff $x = y$). Then
     $\preqsd$ also corresponds to equality, from which it
    follows that the supremum in \eqref{eq:dsig} is attained by $\mu \wedge
    \nu$.  Hence $\alpha_O(\mu, \nu)  = \alpha(\mu, \nu)$.  
\end{proof}

\subsection{Properties of Ordered Affinity}

Let's list some elementary properties of $\alpha_O$.
The following list should be compared with lemma~\ref{l:poaf}.  It shows that
analogous results hold for $\alpha_O$ as hold for $\alpha$.
(Lemma~\ref{l:poaf} is in fact a special case of lemma~\ref{l:poaf2} with the partial
order set to the identity order.)

\begin{lemma}
    \label{l:poaf2}
    For all $(\mu, \nu) \in \mM \times \mM$, we have
    \begin{enumerate}
        \item $0 \leq \alpha_O(\mu, \nu) \leq \min\{\mu (S), \nu(S) \}$,
        \item $\alpha_O(\mu, \nu) = \mu(S) = \nu(S)$ if and only if $\mu \preqsd \nu$,
            and
        \item $c \alpha_O(\mu, \nu) = \alpha_O(c\mu, c\nu)$ whenever $c \geq 0$.
    \end{enumerate}
\end{lemma}

\begin{proof}
    Fix $(\mu, \nu) \in \mM \times \mM$.  Claim (a) follows directly
    from the definitions.  Regarding claim (b),  suppose first that $\mu
    \preqsd \nu$.  Then $(\mu, \nu) \in \Phi(\mu, \nu)$ and hence
    $\alpha_O(\mu, \nu) = \mu(S)$.  Conversely, if $\alpha_O(\mu, \nu) =
    \mu(S)$, then, since the only component $\mu' \leq \mu$ with $\mu'(S) =
    \mu(S)$ is $\mu$ itself, we must have $(\mu, \nu') \in \Phi(\mu, \nu')$
    for some $\nu' \leq \nu$ with $\mu \preqsd \nu'$.  But then $\mu(I) \leq
    \nu'(I) \leq \nu(I)$ for any $I \in i\bB$. Hence $\mu \preqsd \nu$.

    Claim (c) is trivial if $c = 0$, so suppose instead that $c > 0$.
    Fix $(\mu', \nu') \in \Phi(\mu, \nu)$ such that $\alpha_O(\mu, \nu) =
    \mu'(S)$.  It is clear that $(c\mu', c\nu') \in \Phi(c\mu, c\nu)$,
    implying that
    \begin{equation}
        \label{eq:cin}
        c \alpha_O(\mu, \nu) = c \mu'(S) \leq \alpha_O(c\mu, c\nu).
    \end{equation}
    For reverse inequality, we can apply \eqref{eq:cin} again to get
    \begin{equation*}
        \alpha_O(c\mu, c\nu) = c (1/c)  \alpha_O(c\mu, c\nu) \leq c
        \alpha_O(\mu, \nu).
        \qedhere
    \end{equation*}
\end{proof}

\subsection{Equivalent Representations}

In \eqref{eq:affc} we noted that the affinity between two measures
has several alternative representations.  In our setting these results
generalize as follows:

\begin{theorem}
    \label{t:affido}
    For all $(\mu, \nu) \in \pP \times \pP$, we have
    \begin{equation}
        \label{eq:affco}
        \alpha_O(\mu, \nu) 
        = 1 - \sup_{I \in i\bB} \{ \mu(I) - \nu(I) \}
        = \max_{(X, Y) \in \ccC(\mu, \nu)} \PP\{ X \preceq Y \}.
    \end{equation}
\end{theorem}

Evidently \eqref{eq:affc} is a special case of
\eqref{eq:affco} because \eqref{eq:affco}
reduces to \eqref{eq:affc} when $\preceq$ is set to equality.  
For example, when $\preceq$ is equality,
\begin{equation*}
    \sup_{I \in i\bB} \{ \mu(I) - \nu(I) \}
    = \sup_{B \in \bB} \{ \mu(B) - \nu(B) \}
    = \sup_{B \in \bB} | \mu(B) - \nu(B) |.
\end{equation*}
where the last step is from \eqref{eq:sisf}.  Note also that, as shown in
the proof of theorem~\ref{t:affido}, the supremum can also be written in
terms of the open increasing sets $i\oO$ or the closed decreasing sets $d\cC$.
In particular,
\begin{equation*}
    \sup_{I \in i\bB} \{ \mu(I) - \nu(I) \}
    = \sup_{I \in i\oO} \{ \mu(I) - \nu(I) \}
    = \sup_{D \in d\cC} \{ \nu(D) - \mu(D) \}.
\end{equation*}

One of the assertions of theorem~\ref{t:affido} is 
the existence of a coupling $(X, Y) \in \ccC(\mu, \nu)$ attaining
$\PP\{ X \preceq Y \} = \alpha_O(\mu, \nu)$.
Let us refer to any such coupling as an \emph{order maximal} coupling for
$(\mu, \nu)$.

\begin{example}
    For $(x, y) \in S \times S$, we have
    \begin{equation*}
        \alpha_O(\delta_x, \delta_y) = \1\{x \preceq y\} = \1_{\GG}(x, y),
    \end{equation*}
    as can easily be verified from the definition or either of the alternative
    representations in \eqref{eq:affco}.  The map $(x, y) \mapsto \1_{\GG}(x,
    y)$ is measurable due to the Polish assumption.  As a result, for any $(X,
    Y) \in \ccC(\mu, \nu)$ we have 
    \begin{equation*}
        \EE \alpha_O(\delta_X, \delta_Y) 
         = \PP\{X \preceq Y\}
        \leq \alpha_O(\mu, \nu),
    \end{equation*}
    with equality when $(X, Y)$ is an order maximal coupling.
\end{example}

\subsection{Comments on Theorem~\ref{t:affido}}

The existence of an order maximal coupling 
shown in theorem~\ref{t:affido}
implies two well-known results that are usually treated separately.
One is the Nachbin--Strassen theorem
(see, e.g., thm.~1 of \cite{kamae1977stochastic} or ch.~IV of \cite{lindvall2002lectures}), which 
states the existence of a coupling $(X, Y) \in \ccC(\mu, \nu)$ 
attaining $\PP\{ X \preceq Y \} = 1$ whenever $\mu \preqsd \nu$.
The existence of an order maximal coupling for each $(\mu, \nu)$ in $\pP
\times \pP$ implies
this statement, since, under the hypothesis that $\mu \preqsd \nu$, we also have
$\alpha_O(\mu, \nu)= 1$. Hence any order maximal coupling satisfies $\PP\{ X
\preceq Y \} = 1$.

The other familiar result implied by the existence of an order maximal coupling 
is the existence of a maximal coupling in the standard sense (see the discussion
of maximal couplings after \eqref{eq:affc} and the result on p.~19
of \cite{lindvall2002lectures}).  Indeed, if we take $\preceq$ to be the
identity order, then \eqref{eq:affco} reduces to \eqref{eq:affc}, as already discussed.

\section{Total Ordered Variation}

\label{s:tov}

Let $S$ be a partially ordered Polish space.  Consider the function on $\pP \times \pP$ given by
\begin{equation}
    \label{eq:dtovd}
    \gamma(\mu, \nu) 
    := 2  - \alpha_O(\mu, \nu) - \alpha_O(\nu, \mu).
\end{equation}
We call $\gamma(\mu, \nu)$ the \emph{total ordered variation distance} between
$\mu$ and $\nu$.  The natural comparison is with \eqref{eq:tvmtap},
which renders the same value if $\alpha_O$ is replaced by $\alpha$.
In particular, when $\preceq$ is equality, ordered affinity reduces to
affinity, and total ordered variation distance reduces to total variation
distance.
Since ordered affinities dominate affinities (see \eqref{eq:aldbal}), we have
   $\gamma(\mu, \nu) \leq  \| \mu - \nu \|$
   for all $(\mu, \nu) \in \pP \times \pP$.

Other, equivalent, representations of $\gamma$ are available.  For example,
in view of \eqref{eq:affco}, for any $(\mu, \nu) \in \pP \times \pP$ we have
\begin{equation}
    \label{eq:tvsio}
    \gamma(\mu, \nu) 
    = \sup_{I \in i\bB} (\mu - \nu)(I) + \sup_{I \in i\bB} (\nu - \mu)(I),
\end{equation}
By combining lemma~\ref{l:cfs} and \eqref{eq:chho}, we also have
\begin{equation}
    \label{eq:nq}
    2 \gamma(\mu, \nu) 
    = \sup_{h \in iH} (\mu - \nu)(h) + \sup_{h \in iH} (\nu - \mu)(h).
\end{equation}
It is straightforward to show that 
\begin{equation}
    \label{eq:idb}
    \sup_{I \in i\bB} |\mu(I) - \nu(I)| \leq \gamma(\mu, \nu) 
    \quad \text{and} \quad
    \sup_{D \in d\bB} |\mu(D) - \nu(D)| \leq \gamma(\mu, \nu) .
\end{equation}

\begin{lemma}
    The function $\gamma$ is a metric on $\pP$.
\end{lemma}

\begin{proof}
    The claim that $\gamma$ is a metric follows in a straightforward way from the
    definition or the alternative representation \eqref{eq:tvsio}.  For
    example, the triangle inequality is easy to verify using \eqref{eq:tvsio}.
    Also, $\gamma(\mu, \nu) = 0$ implies $\mu = \nu$ by \eqref{eq:tvsio} and
    remark~\ref{r:sfe}.
\end{proof}

\subsection{Connection to Other Modes of Convergence}

As well as total variation, the metric $\gamma$ is closely related to the
so-called \emph{Bhattacharya metric}, which is given by
\begin{equation}
    \label{eq:bhm}
    \beta(\mu, \nu) := \sup_{h \in iH} |\mu(h) - \nu(h)|    .
\end{equation}
See \cite{bhattacharya1988asymptotics, chakraborty1998completeness}.  
(In \cite{chakraborty1998completeness} the metric is defined by taking the
supremum over $iH_0$ rather than $iH$, but the two definitions differ only by
a positive scalar.)   The Bhattacharya metric can be thought of as an
alternative way to generalize total variation distance, in the sense that,
like $\gamma$, the metric $\beta$ reduces to total variation distance when
$\preceq$ is the identity order (since $iH$ equals $H$ under this order).  
From \eqref{eq:cfs2} we have
\begin{equation}
    \label{eq:cfsi}
    \frac{1}{2}
    \left[ 
        \sup_{h \in iH} \lambda(h) +\sup_{h \in iH} (-\lambda)(h)
    \right]
    \leq \sup_{h \in iH} |\lambda(h)|
    \leq \sup_{h \in iH} \lambda(h) + \sup_{h \in iH} (-\lambda)(h),
\end{equation}
and from this and \eqref{eq:nq} we have
\begin{equation}
    \label{eq:em}
    \gamma(\mu, \nu)
    \leq \beta(\mu, \nu) 
    \leq 2 \gamma(\mu, \nu).
\end{equation}
Hence $\beta$ and $\gamma$ are equivalent metrics.

The metric $\gamma$ is also connected to the Wasserstein metric
\cite{gibbs2002choosing, givens1984class}.
If $\rho$ metrizes the topology on $S$, then the Wasserstein distance between
probability measures $\mu$ and $\nu$ is
\begin{equation*}
    w(\mu, \nu) 
    := \inf_{(X, Y) \in \ccC(\mu, \nu)} \EE \, \rho(X, Y) .
\end{equation*}
The total ordered variation metric can be compared as follows. Consider the
''directed semimetric'' $\hat \rho(x, y) := \1\{x \npreceq y\}$.  In view of
\eqref{eq:affco} we have
\begin{equation*}
    \gamma(\mu, \nu)  
    = \inf_{(X, Y) \in \ccC(\mu, \nu)} \EE \, \hat \rho(X, Y) 
        + \inf_{(X, Y) \in \ccC(\mu, \nu)} \EE \, \hat \rho(Y, X) .
\end{equation*}
Thus, $\gamma(\mu, \nu)$ is found by summing
two partial, ``directed Wasserstein deviations.''
Summing the two directed differences from opposite directions yields
a metric.

\begin{proposition}
    \label{p:iwc}
    If $\{\mu_n\}_{n \geq 0} \subset \pP$ is tight and $\gamma(\mu_n, \mu_0)
    \to 0$, then $\mu_n \toweak \mu_0$.
\end{proposition}

\begin{proof}
    Let $\{\mu_n\}$ and $\mu := \mu_0$ satisfy the conditions of the proposition.
    Take any subsequence of $\{\mu_n\}$ and observe that by
    Prohorov's theorem, this subsequence has a subsubsequence
    converging weakly to some $\nu \in \pP$.  
    Along this subsubsequence, 
    for any continuous $h \in ibS$ we have both
    $\mu_n(h ) \to \mu(h )$ and $\mu_n( h ) \to \nu( h )$. This is sufficient for $\nu = \mu$
    by remark~\ref{r:sfe}.
    Thus, every subsequence of $\{\mu_n\}$ has a subsubsequence
    converging weakly to $\mu$, and hence so does the entire sequence.
\end{proof}

\subsection{Completeness}

To obtain completeness of $(\pP, \gamma)$, we adopt the following additional
assumption, 
which is satisfied if, say,
compact sets are order bounded (i.e., lie in order intervals) and order
intervals are compact.  (For example, $\RR^n$ with the usual pointwise partial
order has this property.) 

\begin{assumption}
    \label{a:c}
    If $K \subset S$ is compact, then $i(K) \cap d(K)$ is also compact.
\end{assumption}

\begin{theorem}
    \label{t:bkt}
    If assumption~\ref{a:c} holds, then $(\pP, \gamma)$ is complete.
\end{theorem}

\begin{remark}
    \label{r:chak}
    In \cite{chakraborty1998completeness} it was shown that $\beta$ is a complete
    metric when $S = \RR^n$.  Due to equivalence of the metrics,
    theorem~\ref{t:bkt} extends this result to partially
    ordered Polish spaces where assumption~\ref{a:c} is satisfied.
\end{remark}

\section{Applications}

\label{s:applic}

In this section we show that many results in classical and monotone Markov
chain theory, hitherto treated separately, can be derived from the same set of
results based around total ordered variation and ordered affinity.  

Regarding notation, if $\{S_i\}$ are partially ordered Polish spaces over $i = 0,
1, 2, \ldots$, we often use a common symbol $\preceq$ for the
partial order on any of these spaces.  On 
products of these spaces we use the product topology and pointwise partial
order.  Once again, the symbol $\preceq$ is used
for the partial order.
For example, if $(x_0, x_1)$ and $(y_0, y_1)$ are points in $S_0 \times
S_1$, then $(x_0, x_1) \preceq (y_0, y_1)$ means that $x_0 \preceq y_0$ and
$x_1 \preceq y_1$.  

A function $P \colon (S_0, \bB_1) \to [0, 1]$ is called a
\emph{Markov kernel from $S_0$ to $S_1$} if $x \mapsto P(x, B)$ is
$\bB_0$-measurable for each $B \in \bB_1$ and $B \mapsto P(x, B)$ is in
$\pP_1$ for all $x \in S_0$.  If $S_0 = S_1 = S$, we will
call $P$ a \emph{Markov kernel on $S$}, or just a Markov kernel.
Following standard conventions (see, e.g., \cite{meyn2012markov}), for any Markov kernel $P$ from $S_0$ to $S_1$, any
$h \in bS_1$ and $\mu \in \pP_0$, we define $\mu P \in \pP_1$ and $Ph \in bS_0$ via
\begin{equation*}
    (\mu P)(B) = \int P(x, B) \mu(dx)
    \quad \text{and} \quad
    (Ph)(x) = \int h(y) P(x, dy).
\end{equation*}
Also, $\mu \otimes P$ denotes the joint distribution on $S_0 \times S_1$
defined by
\begin{equation*}
    (\mu \otimes P)(A \times B)
    = \int_A P(x, B) \mu(d x).
\end{equation*}
To simplify notation, we use $P_x$ to represent the measure $\delta_x P
= P(x, \cdot)$.  $P^m$ is the $m$-th composition of $P$ with itself.

\subsection{Order Affinity and Monotone Markov Kernels}

Let $S$ be a Polish space partially ordered by $\preceq$. A Markov kernel $P$ is called
\emph{monotone} if $Ph \in ibS_0$ whenever $h \in ibS_1$.  An equivalent
condition is that $\mu P \preqsd \nu P$ whenever $\mu \preqsd \nu$; or just
$P(x, \cdot) \preqsd P(y, \cdot)$ whenever $x \preceq y$.  It is well-known
(see, e.g., proposition~1 of \cite{kamae1977stochastic}) that if $\mu \preqsd
\nu$ and $P$ is monotone, then $\mu \otimes P \preqsd \nu \otimes P$.  Note
that, when $\preceq$ is the identity order, every Markov kernel is monotone.

\begin{lemma}
    \label{l:mid}
    If $P$ is a monotone Markov kernel from $S_0$ to $S_1$ and $\mu, \mu',
    \nu$ and $\nu'$ are probabilities in $\pP_0$, then
    \begin{equation*}
        \mu' \preqsd \mu 
        \text{ and } 
        \nu \preqsd \nu' 
        \quad \implies \quad
        \alpha_O(\mu P, \nu P) \leq \alpha_O(\mu' P, \nu' P).
    \end{equation*}
\end{lemma}

\begin{proof}
    Let $P, \mu, \mu', \nu$ and $\nu'$ have the stated properties.
    In view of the equivalently representation in \eqref{eq:affco}, the claim
    will be established if 
    \begin{equation*}
        \sup_{I \in i\bB} \{ (\mu P)(I) - (\nu P)(I) \}
        \geq \sup_{I \in i\bB} \{ (\mu' P)(I) - (\nu' P)(I) \}.
    \end{equation*}
    This holds by the monotonicity of $P$ and the 
    order of $\mu, \mu', \nu$ and $\nu'$.
\end{proof}

\begin{lemma}
    \label{l:afup}
    If $P$ is a monotone Markov kernel from $S_0$ to $S_1$, then
    \begin{equation*}
        \alpha_O(\mu P, \nu P) \geq \alpha_O(\mu, \nu)
        \quad \text{for any $\mu, \nu$ in $\pP_0$}.
    \end{equation*}
\end{lemma}

\begin{proof}
    Fix $\mu, \nu$ in $\pP_0$ and
    let $(\hat \mu, \hat \nu)$ be a maximal ordered component pair for $(\mu, \nu)$.
    From monotonicity of $P$ and the fact the Markov kernels preserve the
    mass of measures, it is clear that $(\hat \mu P, \hat \nu P)$ is an ordered
    component pair for $(\mu P, \nu P)$.  Hence
    \begin{equation*}
        \alpha_O(\mu P, \nu P) 
        \geq (\hat \mu P)(S)
        = \hat \mu(S)
        = \alpha_O(\mu, \nu).
        \qedhere
    \end{equation*}
\end{proof}

On the other hand, for the joint distribution, the ordered affinity of the
initial pair is preserved.

\begin{lemma}
    \label{l:afk}
    If $P$ is a monotone Markov kernel from $S_0$ to $S_1$, then
    \begin{equation*}
        \alpha_O(\mu \otimes P, \nu \otimes P) = \alpha_O(\mu, \nu).
        \quad \text{for any $\mu, \nu$ in $\pP_0$}.
    \end{equation*}
\end{lemma}

\begin{proof}
    Fix $\mu, \nu$ in $\pP_0$ and let $(X_0, X_1)$ and $(Y_0, Y_1)$ be 
    random pairs with distributions 
    $\mu \otimes P$ and $\nu \otimes P$ respectively.  We have
    \begin{equation*}
        \PP\{ (X_0, X_1) \preceq (Y_0, Y_1) \}
        \leq \PP\{ X_0 \preceq Y_0 \}
        \leq \alpha_O(\mu, \nu).
    \end{equation*}
    Taking the supremum over all couplings in $\ccC(\mu \otimes P, \nu \otimes
    P)$ shows that $\alpha_O(\mu \otimes P, \nu \otimes P)$ is dominated by
    $\alpha_O(\mu, \nu)$.

    To see the reverse inequality, let $(\hat \mu, \hat \nu)$ be a maximal
    ordered component pair for $(\mu, \nu)$.
    Monotonicity of $P$ now gives $\hat \mu \otimes P \preqsd \hat \nu \otimes
    P$.  Using this and the fact the Markov kernels preserve the
    mass of measures, we see that $(\hat \mu \otimes P, \hat \nu \otimes P)$
    is an ordered component pair for $(\mu \otimes P, \nu \otimes P)$.  Hence
    \begin{equation*}
        \alpha_O(\mu P, \nu P) 
        \geq (\hat \mu \otimes P)(S_0 \times S_1)
        = \hat \mu(S_0)
        = \alpha_O(\mu, \nu).
        \qedhere
    \end{equation*}
\end{proof}

\subsection{Monotone Markov Chains}

Given $\mu \in \pP$ and  Markov kernel $P$ on $S$, a stochastic process $\{X_t\}_{t
\geq 0}$ taking values in $S^{\infty} := \times_{t=0}^\infty S$ will be called
a Markov chain with initial distribution $\mu$ and kernel $P$ if the
distribution of $\{X_t\}$ on $S^{\infty}$ is 
\begin{equation*}
    \bQ_\mu := \mu \otimes P \otimes P \otimes P \otimes \cdots  
\end{equation*}
(The meaning of the right hand side is clarified in, e.g., \S III.8 of
\cite{lindvall2002lectures}, p.~903 of \cite{kamae1977stochastic}, \S 3.4 of
\cite{meyn2012markov}.)  If $P$ is a monotone Markov kernel, then $(x, B) \mapsto
\bQ_x(B) := \bQ_{\delta_x}(B)$ is a monotone Markov kernel from $S$ to $S^\infty$.
See propositions~1 and 2 of \cite{kamae1977stochastic}.

There are various useful results about representations of Markov chains
that are ordered almost surely.  One is that, if the initial conditions
satisfy $\mu \preqsd \nu$ and $P$ is
a monotone Markov kernel, then we can find Markov chains $\{X_t\}$ and
$\{Y_t\}$ with initial distributions $\mu$ and $\nu$ and kernel $P$
such that $X_t \preceq Y_t$ for all $t$ almost surely.  (See, e.g.,
theorem~2 of \cite{kamae1977stochastic}.)  This 
result can be generalized beyond the case where
$\mu$ and $\nu$ are stochastically ordered, using the results presented above.
For example, let $\mu$ and $\nu$ be arbitrary initial distributions and
let $P$ be monotone, so that $\bQ_x$ is likewise monotone.  By
lemma~\ref{l:afk} we have
\begin{equation*}
    \alpha_O(\bQ_\mu, \bQ_\nu) 
    = \alpha_O(\mu \otimes \bQ_x, \nu \otimes \bQ_x) 
    = \alpha_O(\mu, \nu).
\end{equation*}
In other words, the ordered affinity of the entire processes is given by the
ordered affinity of the initial distributions.  It now follows from 
theorem~\ref{t:affido} that there exist Markov chains $\{X_t\}$ and
$\{Y_t\}$ with initial distributions $\mu$ and $\nu$ and Markov kernel $P$
such that
\begin{equation*}
    \PP\{ X_t \preceq Y_t, \; \forall t \geq 0\} = \alpha_O(\mu, \nu).
\end{equation*}
The standard result is a special case, since $\mu \preqsd \nu$ implies
$\alpha_O(\mu, \nu) = 1$, and hence the sequences are ordered almost surely.

\subsection{Nonexpansiveness}

It is well-known that every Markov kernel is nonexpansive with respect to
the total variation norm, so that
\begin{equation}
    \label{eq:netv}
    \| \mu P - \nu P \| \leq \| \mu - \nu \|
    \quad \text{for all }
    (\mu, \nu) \in \pP \times \pP.
\end{equation}
An analogous result is true for $\gamma$ when $P$ is monotone.  That is, 
\begin{equation}
    \label{eq:netov}
    \gamma(\mu P , \nu P ) 
    \leq \gamma( \mu , \nu )
    \quad \text{for all }
    (\mu, \nu) \in \pP \times \pP.
\end{equation}
The bound \eqref{eq:netov} follows directly from lemma~\ref{l:afup}.
Evidently \eqref{eq:netv} be recovered from \eqref{eq:netov} by setting
$\preceq$ to equality.  

Nonexpansiveness is interesting partly in its own right (we apply it in proofs
below) and partly because it suggests that, with some additional assumptions,
we can strengthen it to contractiveness.  We expand on this idea below.

\subsection{An Order Coupling Bound for Markov Chains}

Doeblin \cite{doeblin1938expose} established and exploited the coupling
inequality
\begin{equation}
    \label{eq:d}
    \| \mu_t - \nu_t \| 
    \leq 2 \, \PP\{X_j \not= Y_j \text{ for any } j \leq t\}
\end{equation}
where $\mu_t$ and $\nu_t$ are the time $t$ distributions of Markov chains
$\{X_t\}$ and $\{Y_t\}$ generated by common Markov kernel $P$.
Even when the state space is uncountable, the right hand side of \eqref{eq:d}
can often be shown to converge to zero by manipulating the joint distribution
of $(X_j, Y_j)$ to increase the chance of a meeting  \cite{doob1953stochastic,
    harris1956existence, orey1959recurrent, pitman1974uniform,
thorisson2000coupling, lindvall2002lectures, nummelin2004general,
revuz2008markov,  meyn2012markov}. 

Consider the following generalization: given a monotone Markov kernel $P$ on
$S$ and arbitrary $\mu, \nu \in \pP$, we can construct Markov chains $\{X_t\}$
and $\{Y_t\}$ with common kernel $P$ and respective initial conditions $\mu$
and $\nu$ such that 
\begin{equation}
    \label{eq:do}
    \gamma( \mu_t, \nu_t)
    \leq \PP\{X_j \npreceq Y_j \text{ for any } j \leq t\}
    + \PP\{Y_j \npreceq X_j \text{ for any } j \leq t\}.
\end{equation}
This generalizes \eqref{eq:d} because both the right and left hand
side of \eqref{eq:do} reduces to \eqref{eq:d} if we take $\preceq$ to be
equality. 

To prove \eqref{eq:do}, we need only show that
\begin{equation}
    \label{eq:ido}
    1 - \alpha_O( \mu P^t, \nu P^t) 
    \leq \PP\{X_j \npreceq Y_j \text{ for any } j \leq t\},
\end{equation}
since, with \eqref{eq:ido} is established,
we can reverse the roles of $\{X_t\}$ and $\{Y_t\}$ in \eqref{eq:ido}
to obtain $1 - \alpha_O( \nu P^t, \mu P^t) \leq \PP\{Y_j \npreceq X_j \text{
for any } j \leq t\}$ and then add this inequality to \eqref{eq:ido}
to produce \eqref{eq:do}.

If $\{X_t\}$ and $\{Y_t\}$ are Markov chains with kernel $P$ and initial conditions $\mu$ and
$\nu$, then \eqref{eq:affco} yields $\alpha_O( \mu P^t, \nu P^t) \geq \PP\{X_t
\preceq Y_t\}$.  Therefore, we need only construct such chains with the
additional property that
\begin{equation}
    \label{eq:oia}
    \PP\{X_j \npreceq Y_j \text{ for any } j \leq t\} 
    = \PP\{X_t \npreceq Y_t\}.
\end{equation}
Intuitively we can do so by using a ``conditional'' version of the Nachbin--Strassen theorem,
producing chains that, once ordered, remain ordered almost surely.  
This can be formalized as follows: By \cite[theorem~2.3]{machida2010monotone}, there exists a
Markov kernel $M$ on $S \times S$ such that $\GG$ is absorbing for $M$ (i.e.,
$M((x,y), \GG) = 1$ for all $(x,y)$ in $\GG$),
\begin{equation*}
    P(x, A) = M((x,y), \, A \times S)
    \quad \text{and} \quad
    P(y, B) = M((x,y), \, S \times B)
\end{equation*}
for all $(x,y) \in S \times S$ and all $A, B \in \bB$.  Given $M$, let $\eta$
be a distribution on $S \times S$ with marginals $\mu$ and $\nu$, let 
$\bQ_\eta := \eta \otimes M \otimes M \otimes \cdots$ be the induced joint
distribution, and let $\{(X_t, Y_t)\}$ have
distribution $\bQ_\eta$ on $(S \times S)^{\infty}$.  By construction, $X_t$
has distribution $\mu P^t$ and $Y_t$
has distribution $\nu P^t$.  Moreover, \eqref{eq:oia} is valid because $\GG$
is absorbing for $M$, and hence 
$\PP\{(X_j, Y_j) \notin \GG \text{ for any } j \leq t\} = \PP\{(X_t,  Y_t)
\notin \GG\}$.

\subsection{Uniform Ergodicity}

\label{ss:ue}

Let $S$ be a partially ordered Polish space satisfying assmption~\ref{a:c},
and let $P$ be a monotone Markov kernel on $S$.  A distribution $\pi$ is
called \emph{stationary} for $P$ if $\pi P = \pi$.  Consider the value
\begin{equation}
    \label{eq:odob}
    \sigma(P)
    := \inf_{(x, y) \in S \times S}  \alpha_O(P_x, P_y) ,
\end{equation}
which can be understood as an order-theoretic 
extension of the Markov--Dobrushin coefficient of ergodicity
\cite{dobrushin1956central, seneta1979coefficients}.  It reduces to the usual 
notion when $\preceq$ is equality.

\begin{theorem}
    \label{t:ue}
    If $P$ is monotone, then
    \begin{equation}
        \label{eq:fbod}
        \gamma(\mu P, \nu P) \leq (1 - \sigma(P)) \, \gamma(\mu, \nu)
        \quad \text{for all }
        (\mu, \nu) \in \pP \times \pP. 
    \end{equation}
\end{theorem}

Thus, strict positivity of $\sigma(P)$ implies that $\mu \mapsto \mu P$ is a
contraction map on $(\pP, \gamma)$.  
Moreover, in many settings, the bound in
\eqref{eq:fbod} cannot be improved upon. For example,

\begin{lemma}
    \label{l:cbi}
    If $P$ is monotone, $S$ is not a singleton and any $x,y$ in
    $S$ have a lower bound in $S$, then 
    \begin{equation}
        \label{eq:ev}
        \forall \; \xi > \sigma(P), \; \;
        \exists \; \mu, \nu \in \pP
        \; \text{ s.t. } \;
        \gamma(\mu P, \nu P) > (1 - \xi) \, \gamma(\mu, \nu).
    \end{equation}
\end{lemma}

The significance of theorem~\ref{t:ue} is summarized in the next
corollary.

\begin{corollary} 
    \label{c:ue}
    Let $P$ be monotone and let $S$ satisfy assumption~\ref{a:c}.  If there
    exists an $m \in \NN$ such that $\sigma(P^m) > 0$, then $P$ has a unique
    stationary distribution $\pi$ in $\pP$, and
    \begin{equation}
        \label{eq:ifbod}
        \gamma(\mu P^t, \pi) 
        \leq (1 - \sigma(P^m))^{\floor*{t/m}} \, \gamma(\mu, \pi) 
        \quad \text{ for all } \mu \in \pP, \; t \geq 0. 
    \end{equation}
\end{corollary}

Here $\floor*{x}$ is the largest $n \in \NN$ with $n \leq x$.

\begin{proof}
    Let $P$ and $\mu$ be as in the statement of the theorem.  The existence of
    a fixed point of $\mu \mapsto \mu P$, and hence a stationary distribution
    $\pi \in \pP$, follows from theorem~\ref{t:ue} applied to $P^m$,
    Banach's contraction mapping theorem, and the completeness
    of $(\pP, \gamma)$ shown in theorem~\ref{t:bkt}.
    The bound in \eqref{eq:ifbod} follows from \eqref{eq:fbod} applied to
    $P^m$ and the nonexpansiveness of $P$ in the metric $\gamma$ (see
    \eqref{eq:netov}).
\end{proof}

\subsection{Applications from the Introduction}

Let us see how corollary~\ref{c:ue} can be used to show stability of the two
models discussed in the introduction, beginning with the monotone model in
\eqref{eq:bern}. The state space is $S=[0, 1]$ with its standard order and all
assumptions are as per the discussion immediately following \eqref{eq:bern}.
We let $P$ be the corresponding Markov kernel and consider the following coupling of $(P_1, P_0)$:
Let $W$ be a draw from the Bernoulli$(1/2)$ distribution and let $V = 1 - W$.
Then $P_1$ and $P_0$ are, respectively, the distributions of
\begin{equation*}
    X := \frac{1 + W}{2} 
    \quad \text{and} \quad
    Y := \frac{0 + V}{2} = \frac{1 - W}{2}.
\end{equation*}
Since $X \leq Y$ if and only if $W=0$, we have, by lemma~\ref{l:mid} and \eqref{eq:affco}, 
\begin{equation*}
    \alpha_O(P_x, P_y) 
    \geq \alpha_O(P_1, P_0) 
    \geq \PP\{ X \preceq Y \} = \frac{1}{2}
\end{equation*}
for all $x, y \in S$.  From the definition in \eqref{eq:odob} we then have
$\sigma(P) \geq 1/2$, and globally stability in the $\gamma$ metric follows
from corollary~\ref{c:ue}.

Next we turn to the inventory model in \eqref{eq:invent}, with state space
$S=[0, K]$ and $\{W_t\}$ an {\sc IID} process satisfying $\PP\{W_t \geq w\} >
0$ for any real $w$.  Let $\{X_t\}$ and $\{Y_t\}$ be generated by
\eqref{eq:invent}, with common shock sequence $\{W_t\}$, and respective
initial conditions $x,y$ in $S$.  In view of \eqref{eq:affc}, we have
\begin{equation*}
    \alpha(P_x, P_y) 
    \geq \PP\{X_1 = Y_1\}
    \geq \PP\{W_1 \geq K\} =: \kappa > 0.
\end{equation*}
Letting $\preceq$ be equality on $S$, we have $\alpha_O(P_x, P_y) =
\alpha(P_x, P_y) \geq \kappa$, and hence $\sigma(P) \geq \kappa$.
Globally stability in the $\gamma$ metric follows from corollary~\ref{c:ue}.

\subsection{General Applications}

Next we show that theorem~\ref{t:ue} and in particular corollary~\ref{t:ue}
cover as special cases two well known results on Markov chain stability from
the classical literature on one hand and the more recent monotone Markov chain
literature on the other.

Consider first the standard notion of uniform ergodicity:  A Markov kernel $P$ on
$S$ is called \emph{uniformly ergodic} if it has a stationary distribution
$\pi$ and $\sup_{x \in S} \| P_x^t - \pi \| \to 0$ as $t \to \infty$.  Uniform
ergodicity was studied by Markov \cite{markov1906extension} in a countable
state space and by Doeblin \cite{doeblin1938proprietes},  Yoshida and Kakutani
\cite{yoshida1941operator}, Doob \cite{doob1953stochastic} and many subsequent
authors in a general state space. It is defined and reviewed in chapter 16 of
\cite{meyn2012markov}.  One of the most familiar equivalent conditions for
uniform ergodicity \cite[thm.~16.0.2]{meyn2012markov} is the existence of an
$m \in \NN$ and a nontrivial $\phi \in \mM$ such that $P^m_x \geq  \phi$ for
all $x$ in $S$.

One can recover this result using
corollary~\ref{c:ue}.  Take $\preceq$ to be equality, in which
case every Markov operator is monotone, $\gamma$ is total variation distance
and assumption~\ref{a:c} is always satisfied.  Moreover, $\sigma(P^m)$ reduces to
the ordinary ergodicity coefficient of $P^m$, evaluated using the standard
notion of affinity, and hence
\begin{equation*}
    \sigma(P^m) 
    = \inf_{(x, y) \in S \times S}  \alpha(P^m_x, \, P^m_y)
    = \inf_{(x, y) \in S \times S}  (P^m_x \wedge P^m_y) (S)
    \geq \phi(S) > 0.
\end{equation*}
Thus, all the conditions of corollary~\ref{c:ue} are satisfied, and
\begin{equation*}
    \sup_{x \in S}
        \| P^t_x - \pi \| 
    = \sup_{x \in S}
        \gamma( P^t_x \, ,  \, \pi) 
        \leq 2 (1 - \sigma(P^m))^{\floor*{t/m}} \to 0
    \qquad (t \to \infty).
\end{equation*}

Now consider the setting of Bhattacharya and Lee
\cite{bhattacharya1988asymptotics}, where $S = \RR^n$, $\preceq$ is the usual
pointwise partial order $\leq$ for vectors,
and $\{g_t\}$ is a sequence of {\sc iid} random maps from $S$ to itself,
generating $\{X_t\}$ via $X_t = g_t (X_{t-1}) =
g_t \circ \cdots \circ g_1 (X_0)$.  The corresponding Markov kernel
is $P(x,B) = \PP\{g_1(x) \in B\}$.  
The random maps are assumed to be order preserving on $S$, so that $P$ is
monotone.  Bhattacharya and Lee use a ``splitting condition,'' which
assumes existence of a $\bar x \in S$ and $m \in \NN$ such that
\begin{enumerate}
    \item $s_1 := 
    \PP \{ g_m \circ \cdots \circ g_1 (y) 
        \leq \bar x, \; \forall y \in S\} > 0$ and
    \item $s_2 :=
        \PP \{ g_m \circ \cdots \circ g_1 (y) \geq \bar
        x, \; \forall y \in S\} > 0$.
\end{enumerate}
Under these assumptions, they show that 
$\sup_{x \in S} \beta(P^t_x, \, \pi)$ converges to zero exponentially fast in
$t$, where $\beta$ is the Bhattacharya metric introduced in \eqref{eq:bhm}.
This finding extends earlier results by Dubins and Freedman
\cite{dubins1966invariant} and Yahav \cite{yahav1975fixed} to multiple
dimensions.

This result can be obtained as a special case of corollary~\ref{c:ue}.  
Certainly $S$ is a partially ordered Polish space and
assumption~\ref{a:c} is satisfied.  Moreover, the ordered ergodicity
coefficient $\sigma(P^m)$ is strictly positive.  To see this, suppose that the
splitting condition is satisfied at $m \in \NN$. Pick any $x, y \in S$
and let $\{X_t\}$ and $\{Y_t\}$ be independent copies of the Markov chain,
starting at $x$ and $y$ respectively.  We have
\begin{equation*}
    \sigma(P^m)
     \geq \PP\{X_m \leq Y_m\} 
     \geq \PP\{X_m \leq \bar x \leq Y_m\} 
     = \PP\{X_m \leq \bar x \} \PP\{\bar x \leq Y_m\} \geq s_1 s_2.
\end{equation*}
The last term is strictly positive by assumption.  Hence all the conditions of
corollary~\ref{c:ue} are satisfied, a unique stationary distribution $\pi$
exists, and $\sup_{x \in S} \gamma(P^t_x, \, \pi)$
converges to zero exponentially fast in $t$.  We showed in \eqref{eq:em} that
$\beta \leq 2\gamma$, so the same convergence holds for the Bhattacharya metric.

We can also recover a related convergence result due to Hopenhayn and
Prescott \cite[theorem~2]{hopenhayn1992stochastic}
that is routinely applied to stochastic stability problems in economics.  They assume that $S$ is a compact metric
space with a closed partial order and a least element $a$ and greatest element
$b$.  They suppose that $P$ is monotone, and that there exists an $\bar x$ in $S$ and an $m \in \NN$ such that 
\begin{equation}
    \label{eq:mmc}
    P^m(a, [\bar x, b]) > 0
    \quad \text{and} \quad
    P^m(b, [a, \bar x]) > 0.
\end{equation}
In this setting, they show that $P$ has a unique
stationary distribution $\pi$ and $\mu P^t \toweak \pi$
for any $\mu \in \pP$ as $t \to \infty$.  
This result can be obtained from corollary~\ref{c:ue}.  Under the stated
assumptions, $S$ is Polish and assumption~\ref{a:c} is satisfied.  
The coefficient $\sigma(P^m)$ is strictly positive because, if we
let $\{X_t\}$ and $\{Y_t\}$ be independent copies of the Markov chain
starting at $b$ and $a$ respectively, then, since $(X_m, Y_m) \in \ccC(P^m_b,
P^m_a)$, we have
\begin{equation*}
    \alpha_O(P^m_b, P^m_a)
     \geq \PP\{X_m \leq Y_m\} 
     \geq \PP\{X_m \leq \bar x \leq Y_m\} 
     = \PP\{X_m \leq \bar x \} \PP\{\bar x \leq Y_m\}.
\end{equation*}
The last term is strictly positive by \eqref{eq:mmc}.
Positivity of $\sigma(P^m)$ now follows from 
lemma~\ref{l:mid}, since $a \preceq x, y \preceq b$ for all $x, y \in S$.
Hence, by corollary~\ref{c:ue},  there exists a unique
stationary distribution $\pi$ and $\gamma(\mu P^t, \pi) \to 0$ 
as $t \to \infty$ for any $\mu \in \pP$.  This convergence implies weak
convergence by proposition~\ref{p:iwc} and compactness of $S$.

\section{Appendix}

The appendix collects remaining proofs.  Throughout, in addition to notation
defined above, $cbS_0$ denotes all continuous
functions $h \colon S \to [0, 1]$, while 
\begin{equation*}
    g(\mu, \nu) := \| \mu \| - \alpha_O(\mu, \nu)
    \quad \text{for each $\mu, \nu \in \mM$}.
\end{equation*}

\subsection{Proofs of Section~\ref{s:otg} Results}

\begin{proof}[Proof of lemma~\ref{l:cfs}]
    For the first equality, fix $\lambda \in \mM_s$ and let
    \begin{equation*}
        s(\lambda) := \sup_{I \in i\bB} \lambda(I) 
        \quad \text{and} \quad
        b(\lambda) := \sup_{h \in iH_0} \lambda(h).
    \end{equation*}
    Since $\1_{I} \in ibS$ for all $I \in i\bB$, we have $b(\lambda) \geq s(\lambda)$.
    To see the reverse inequality, let $h \in iH_0$.  Fix $n \in \NN$.  Let
        $r_{j} := j/n$ for $j = 0, \ldots, n$.  Define $h_{n} \in iH_0$ by
    \begin{equation*} 
        h_{n}(x) = \max\{r \in \{r_{0}, \ldots, r_{n}\} : r \leq h(x)\}.
    \end{equation*}
    Since $h \leq h_{n} + 1/n$, we have
    \begin{equation} 
        \label{calpie}
        \lambda (h) \leq \lambda (h_{n}) + \frac{\|\lambda\|}{n}.
    \end{equation}
    For $j = 0, \ldots, n$, let $I_{j} := \{x \in S : h_{n}(x) \geq r_{j}\} \in i\bB$.
    Note that
    \begin{equation} 
        \label{capable}
        I_{n} = \{x \in S : h_{n}(x) = 1\} \subset I_{n-1} \subset \cdots
        \subset I_{0} = S.
    \end{equation}
    We have
        $\lambda (h_{n} )
        = \lambda(I_{n}) + \sum_{j=1}^{n} r_{n-j} \lambda(I_{n-j} 
        \setminus I_{n-j+1})$.
    We define $f_{0}, \ldots, f_{n-1} \in iH_0$ and $A_{0}, \ldots, A_{n-1}
    \in i\bB$ as follows.  Let $f_{0} = h_{n}$ and $A_{0} = I_{n}$.  Evidently
    \begin{equation*} 
        \lambda (f_{0}) \geq \lambda (h_{n}), 
        \quad \A x \in A_{0}, \; f_{0}(x) = 1, 
        \quad \A x \in I_{n-1} \setminus A_{0}, \; f_{0}(x) = r_{n-1}.
    \end{equation*}
    Now suppose that for some $j \in \{0, 1, \ldots, n-2\}$, we have
    \begin{equation} 
        \label{calfeutrage}
        \lambda (f_{j}) \geq \lambda (h_{n}), 
        \quad \A x \in A_{j}, \; f_{j}(x) = 1,
        \quad \A x \in I_{n-j-1} \setminus A_{j}, \; f_{j}(x) = r_{n-j-1}.
    \end{equation}
    If $\lambda(I_{n-j-1} \setminus A_{j}) > 0$, then define
    \begin{equation*}
        f_{j+1}(x) =
        \begin{cases}
            1 & \text{if $x \in I_{n-j-1} \setminus A_{j}$,}
            \\
            f_{j}(x) & \text{otherwise,}
        \end{cases}
        \qquad \text{and} \qquad
        A_{j+1} = I_{n-j-1}.
    \end{equation*}
    Note that in this case
    \begin{gather*} 
        \lambda (f_{j+1}) - \lambda (f_{j} )
        = (1-r_{n-j-1}) \lambda(I_{n-j-1} \setminus A_{j}) > 0,
        \\ \label{calfater}
        \A x \in I_{n-j-2} \setminus A_{j+1}, \quad f_{j+1}(x) = r_{n-j-2}.
    \end{gather*}
    If $\lambda(I_{n-j-1} \setminus A_{j}) \leq 0$, then define
    \begin{equation*}
        f_{j+1}(x) =
        \begin{cases}
            r_{n-j-2} & \text{if $x \in I_{n-j-1} \setminus A_{j}$,}
            \\
            f_{j}(x) & \text{otherwise,}
        \end{cases}
        \qquad \text{and} \qquad
        A_{j+1} = A_{j}.
    \end{equation*}
    In this case 
        $\lambda (f_{j+1}) - \lambda (f_{j} )
        = (r_{n-j-2} - r_{n-j-1}) \lambda(I_{n-j-1} \setminus A_{j}) \geq 0$.
    We also have (\ref{calfater}) by construction.  Thus in both cases, we
    have (\ref{calfeutrage}) with $j$ replaced by $j+1$.  Continuing this way,
    we see that (\ref{calfeutrage}) holds for all $j = 0, \ldots, n-1$.

    Let $j = n-1$ in (\ref{calfeutrage}).  From the definition of $r_j$ and
    (\ref{capable}) we have $r_{0} = 0$ and $I_{0} = S$.  Thus
    \begin{equation*} 
        \lambda (f_{n-1}) \geq \lambda (h_{n}), 
        \quad \A x \in A_{n-1}, \; f_{n-1}(x) = 1,
        \quad \A x \in S \setminus A_{n-1}, \; f_{n-1}(x) = 0.
    \end{equation*}
    Since $f_{n-1} = \1_{A_{n-1}}$ and $A_{n-1} = I_{j}$ for some $j \in \{0,
    \ldots, n-1\}$, recalling (\ref{calpie}) we have
    \begin{equation*}
        \lambda (h)  - \frac{\|\lambda\|}{n} 
        \leq \lambda (h_{n} )
        \leq \lambda(A_{n-1}) 
        \leq s(\lambda).
    \end{equation*}
    Applying $\sup_{h \in iH_0}$ to the leftmost side, we see that $b(\lambda)
    - 1/n \leq s(\lambda)$.  Since this is true for any $n \in \NN$, we obtain
    $b(\lambda) \leq s(\lambda)$.

    The claim \eqref{eq:cfs2} follows from $|a| = \max\{a, -a\}$ and 
    interchange of $\max$ and $\sup$.
\end{proof}

\begin{proof}[Proof of theorem~\ref{t:affido}]
    Let $(X,Y)$ be a coupling of $(\mu, \nu)$, and define
    \begin{equation*}
        \mu'(B) := \PP\{X \in B,\, X \preceq Y\}
        \quad \text{and} \quad
        \nu'(B) := \PP\{Y \in B,\, X \preceq Y\}.
    \end{equation*}
    Clearly $\mu' \leq \mu$, $\nu' \leq \nu$ and $\mu'(S) = \PP\{X \preceq Y\}
    = \nu'(S)$.  Moreover, for any increasing set $I \in \bB$ we clearly have
        $\mu'(I) = \nu'(I)$.
    Hence $(\mu', \nu') \in \Phi(\mu, \nu)$ and $\PP\{X \preceq Y\} =
    \mu'(S) \leq \alpha_O(\mu, \nu)$.  We now exhibit a coupling such that
    equality is attained. In doing so, 
    we can assume that $a := \alpha_O(\mu, \nu) > 0$.\footnote{If not,
        then for any $(X,Y) \in \ccC(\mu, \nu)$ we have $0 \leq \PP\{X \preceq
        Y\} \leq \alpha_O(\mu, \nu) = 0$.}  

    To begin, observe that, by proposition~\ref{p:em}, there exists a pair
    $(\mu', \nu') \in \Phi(\mu, \nu)$ with $\mu'(S) = \nu'(S) = a$.  Let
        $\mu^r := \frac{\mu - \mu'}{1 - a}$
        and $\nu^r := \frac{\nu - \nu'}{1 - a}$.
    By construction, $\mu^r, \nu^r, \mu' / a$ and $\nu' / a$ are probability measures
    satisfying
    \begin{equation*}
        \mu = (1 - a) \mu^r + a (\mu'/a)
        \quad \text{and} \quad
        \nu = (1 - a) \nu^r + a (\nu'/ a).
    \end{equation*} 
    We construct a coupling $(X,Y)$ as follows.  Let $U$, $X'$, $Y'$, $X^r$
    and $Y^r$ be random variables on a common probability space
    such that
    \begin{enumerate}
        \item[(a)] $X' \disteq \mu' / a$, $Y' \disteq \nu' / a$, $X^r
            \disteq \mu^r$ and $Y^r \disteq \nu^r$
        \item[(b)] $U$ is uniform on $[0,1]$ and independent of $(X',Y', X^r, Y^r)$ and
        \item[(c)]$\PP\{X' \preceq Y'\} = 1$.
    \end{enumerate}
    The pair in (c) can be constructed via the Nachbin--Strassen theorem 
    \cite[thm.~1]{kamae1977stochastic}, since $\mu' / a \preqsd
    \nu' / a$.  Now let
    \begin{equation*}
        X := \1\{U \leq a\} X' + \1\{U > a\} X^r
        \quad \text{and} \quad
        Y := \1\{U \leq a\} Y' + \1\{U > a\} Y^r.
    \end{equation*}
    Evidently $(X,Y) \in \ccC(\mu, \nu)$.  Moreover, for this pair,
    we have
    \begin{equation*}
        \PP\{X \preceq Y\} 
        \geq \PP\{X \preceq Y, \; U \leq a \} 
        = \PP\{X' \preceq Y', \; U \leq a \} .
    \end{equation*}
    By independence the right hand side is equal to $\PP\{X' \preceq Y'\}
    \PP\{U \leq a \} = a$, so $\PP\{X \preceq Y\} \geq a := \alpha_O(\mu, \nu)$. 
    We conclude that
    \begin{equation}
        \label{eq:affosmo}
        \alpha_O(\mu, \nu) =
        \max_{(X, Y) \in \ccC(\mu, \nu)} \PP\{ X \preceq Y \}.
    \end{equation}
    Next, observe that, for any $(X, Y) \in \ccC(\mu, \nu)$ and $h \in ibS$, we have
    \begin{align*}
        \mu(h) - \nu( h)
        & = \EE h(X) - \EE h(Y)
        \\
        & = \EE [ h(X) - h(Y)] \1\{ X \preceq Y \} + \EE [ h(X) - h(Y)] \1\{ X \npreceq Y \} 
        \\
        & \leq \EE [ h(X) - h(Y)] \1\{ X \npreceq Y \} .
    \end{align*}
    Specializing to $h = \1_I$ for some $I \in i\bB$, we have
        $\mu(I) - \nu(I) \leq \PP\{ X \npreceq Y \} = 1 - \PP\{ X \preceq Y \}$.
    From this bound and \eqref{eq:affosmo}, 
    the proof of \eqref{eq:affco} will be complete if we can show that 
    \begin{equation}
        \label{eq:ci3}
        \sup_{(X, Y) \in \ccC(\mu, \nu)} \PP\{X \preceq Y\}
        \geq 
        1 - \sup_{I \in i\bB} \{ \mu(I) - \nu(I) \}.
    \end{equation}
    To prove \eqref{eq:ci3}, let 
    $\bB \otimes \bB$ be the product $\sigma$-algebra on $S \times S$ and  let
    $\pi_i$ be the $i$-th coordinate projection, so that $\pi_1(x,y) = x$ and
    $\pi_2(x,y) = y$ for any $(x,y) \in S \times S$.  As usual, given $Q \subset S
    \times S$, we let $\pi_1(Q)$ be all $x \in S$ such that $(x,y) \in Q$, and
    similarly for $\pi_2$.   Recall that $\cC$ is the closed sets in $S$ and
    $d\cC$ is the decreasing sets in $\cC$.
    Strassen's theorem \cite{strassen1965existence} implies
    that, for any $\epsilon \geq 0$ and any closed set $K
    \subset S \times S$, there exists a probability measure $\xi$ on $(S \times S, \bB
    \otimes \bB)$ with marginals $\mu$ and $\nu$ such that $\xi(K) \geq 1
    - \epsilon$ whenever
    \begin{equation*}
        \nu(F) \leq \mu(\pi_1(K \cap (S \times F))) + \epsilon,
        \qquad \forall \,  F \in \cC. 
    \end{equation*}
    Note that if $F \in \cC$, then, since $\preceq$ is a closed partial order,
    so is the smallest decreasing set $d(F)$ that containts $F$.
    Let $\epsilon := \sup_{D \in d\cC} \{ \nu(D) - \mu(D) \}$, so that
    \begin{equation*}
        \epsilon 
        \geq \sup_{F \in \cC} \{ \nu(d(F)) - \mu(d(F)) \} 
        \geq \sup_{F \in \cC} \{ \nu(F) - \mu(d(F)) \}.
    \end{equation*}
    Noting that $d(F)$ can be expressed as $\pi_1(\GG \cap (S \times F))$, it
    follows that, for any $F \in \cC$,
    \begin{equation*}
        \nu(F) \leq \mu(\pi_1(\GG \cap (S \times F))) + \epsilon.
    \end{equation*}
    Since $\preceq$ is closed, $\GG$ is closed, and Strassen's theorem
    applies.  From this theorem we obtain a probability measure $\xi$ on the
    product space $S \times S$ such that $\xi(\GG) \geq 1 - \epsilon$ and
    $\xi$ has marginals $\mu$ and $\nu$.  

    Because complements of increasing sets are decreasing and vice versa, we have
    \begin{equation}
        \label{eq:fis}
        \sup_{I \in i\bB} \{ \mu(I) - \nu(I) \}
        \geq \sup_{D \in d\cC} \{ \nu(D) - \mu(D) \}
        = \epsilon
        \geq 1 - \xi(\GG).
    \end{equation}
    Now consider the probability space $(\Omega, \fF, \PP) = (S \times S, \bB
    \otimes \bB, \xi)$, and let $X = \pi_1$ and $Y = \pi_2$.  We then have
        $\xi(\GG) 
        = \xi \setntn{(x,y) \in S \times S}{x \preceq y}
        = \PP \{X \preceq Y\}$.
    Combining this equality with \eqref{eq:fis} implies \eqref{eq:ci3}.
\end{proof}

\subsection{Proofs of Section~\ref{s:tov} Results}

We begin with an elementary lemma:

\begin{lemma}
    \label{cal}
    For any $\mu, \nu \in \mM$, we have
    $\mu \leq \nu$ whenever $\mu(h) \leq \nu (h)$ for all $h \in cbS_0$.
\end{lemma}

\begin{proof}
    Suppose that $\mu (h) \leq \nu (h)$ for all $h \in cbS_0$.  We claim that
    \begin{equation} \label{calamite}
        \text{$\mu(F) \leq \nu(F)$ for any closed set $F \subset S$.}
    \end{equation}
    To see this, let $\rho$ be a metric compatible with the topology of $S$
    and let $F$ be any closed subset of $S$.  Let
        $f_{\epsilon}(x) 
        := \max\{1-\rho(x, F)/\epsilon,\, 0\}$ for  
        $\epsilon > 0, \, x \in S$,
    where $\rho(x, F) = \inf_{y \in F} \rho(x,y)$.  Since $\rho(\cdot, F)$ is
    continuous and $0 \leq
    f_{\epsilon} \leq 1$, we have  $f_{\epsilon} \in cbS_0$.  Let
    $F_{\epsilon} = \{x \in S : \rho(x, F) < \epsilon\}$ for $\epsilon > 0$.
    Note that $f_{\epsilon}(x) = 1$ for all $x \in F$, and that
    $f_{\epsilon}(x) = 0$ for all $x \not\in F_{\epsilon}$.  Thus,
    \begin{equation} \label{calamine}
        \mu(F) 
        \leq \mu (f_{\epsilon}) 
        \leq \nu (f_{\epsilon}) 
        \leq \nu(F_{\epsilon}).
    \end{equation}
    Since $F = \cap_{\epsilon > 0} F_{\epsilon}$, we have $\lim_{\epsilon
    \downarrow 0} \nu(F_{\epsilon}) = \nu(F)$, so letting $\epsilon
    \downarrow 0$ in (\ref{calamine}) yields $\mu(F) \leq \nu(F)$. Hence
    (\ref{calamite}) holds.

    Let $B \in \bB$ and fix $\epsilon > 0$.  Since all probability measures on
    a Polish space are regular, there exists a closed set $F
    \subset B$ such that $\mu(B) \leq \mu(F) + \epsilon$.  Thus by
    (\ref{calamite}), we have
        $\mu(B) \leq \mu(F) + \epsilon \leq \nu(F) + \epsilon \leq \nu (B) +
        \epsilon$.
    Since $\epsilon > 0$ is arbitrary, this yields $\mu(B) \leq \nu(B)$.
    Hence $\mu \leq \nu$.
\end{proof}

\begin{proof}[Proof of theorem~\ref{t:bkt}]
    Let $\{\mu_n\}$ be a Cauchy sequence in $(\pP, \gamma)$.  Our first claim
    is that $\{\mu_n\}$ is tight.  To show this, fix $\epsilon > 0$.  Let $\mu :=
    \mu_N$ be such that 
    \begin{equation}
        \label{eq:ngn}
        n \geq N \implies \gamma(\mu, \mu_n) < \epsilon.
    \end{equation}
    Let $K$ be a compact subset of $S$ such that $\mu(K) \geq 1 - \epsilon$
    and let $\bar K := i(K) \cap d(K)$.  We have
    \begin{equation*}
        \mu_n(\bar K^c)
        = \mu_n( i(K)^c \cup d(K)^c )
        \leq \mu_n( i(K)^c ) + \mu_n(d(K)^c ).
    \end{equation*}
    For $n \geq N$, this bound, \eqref{eq:idb}, \eqref{eq:ngn} and the
    definition of $K$ yield
    \begin{equation*}
        \mu_n(\bar K^c)
        < \mu( i(K)^c ) + \mu(d(K)^c ) + 2\epsilon
        \leq \mu( K^c ) + \mu(K^c ) + 2\epsilon 
        \leq 4 \epsilon.
    \end{equation*}
    Hence $\{\mu_n\}_{n \geq N}$ is tight. It follows that 
    $\{\mu_n\}_{n \geq 1}$ is likewise tight.
    As a result, by Prohorov's theorem, it has a subsequence that converges weakly
    to some $\mu^{*} \in \pP$.  We aim to show that $\gamma(\mu_{n},
    \mu^{*}) \rightarrow 0$.

    To this end, fix $\epsilon > 0$ and let $n_{\epsilon}$ be such that
    $\gamma(\mu_m, \mu_{n_\epsilon}) < \epsilon$ whenever $m \geq n_\epsilon$.
    Fix $m \geq n_{\epsilon}$ and let $\nu := \mu_{m}$.
    For all $n \geq n_{\epsilon}$, we have $\gamma(\nu, \mu_{n}) < \epsilon$.
    Fixing any such $n \geq n_{\epsilon}$, we observe that
    since $g(\mu_{n}, \nu) < \epsilon$, there exists $(\tilde{\mu}_{n},
    \tilde{\nu}_{n}) \in \Phi(\mu_{n}, \nu)$ with $\|\tilde{\mu}_{n}\| =
    \|\tilde{\nu}_{n}\| > 1-\epsilon$.  Multiplying $\tilde{\mu}_{n}$ and
    $\tilde{\nu}_{n}$ by $(1-\epsilon)/\|\tilde{\mu}_n\| < 1$, denoting the
    resulting measures by $\tilde{\mu}_{n}$ and $\tilde{\nu}_{n}$ again, we have
    \begin{equation} \label{calembour}
        \tilde{\mu}_{n} 
        \leq \mu_{n}, \quad \tilde{\nu}_{n} 
        \leq \nu, \quad \|\tilde{\mu}_{n}\| 
        = \|\tilde{\nu}_{n}\| 
        = 1-\epsilon, \quad \tilde{\mu}_{n} \preqsd \tilde{\nu}_{n}.
    \end{equation}
    Note that $\{\tilde{\nu}_{n}\}$ is tight.  Since $\{\mu_{n}\}$ is tight, so is
    $\{\tilde{\mu}_{n}\}$.  Thus there exist subsequences $\{\mu_{n_{i}}\}_{i \in
    \NN}, \{\tilde{\mu}_{n_{i}}\}_{i \in \NN}$, and $\{\tilde{\nu}_{n_{i}}\}_{i \in
    \NN}$ of $\{\mu_{n}\}, \{\tilde{\mu}_{n}\}$, and $\{\tilde{\nu}_{n}\}$
    respectively such
    that, for some $\tilde{\mu}^{*},  \tilde{\nu}^{*} \in \mM$ with
        $\|\tilde{\mu}^{*}\| = \|\tilde{\nu}^{*}\| = 1-\epsilon$,
    we have
    \begin{equation*}
        \mu_{n_{i}} \toweak \mu^{*}, 
        \quad \tilde{\mu}_{n_{i}} \toweak \tilde{\mu}^{*}, 
        \quad \tilde{\nu}_{n_{i}} \toweak \tilde{\nu}^{*}, 
        \quad \A i \in \NN,\; \tilde{\mu}_{n_{i}} \preqsd \tilde{\nu}_{n_{i}}.
    \end{equation*}
    Given $h \in cbS_0$, since $\tilde{\mu}_{n_{i}} (h) \leq \mu_{n_{i}} (h)$ and
    $\tilde{\nu}_{n_{i}} (h) \leq \nu (h)$ for all $i \in \NN$ by (\ref{calembour}), we have
    $\tilde{\mu}^{*} (h) \leq  \mu^{*} (h)$ and $\tilde{\nu}^{*} (h) \leq \nu
    (h)$ by weak
    convergence.  Thus $\tilde{\mu}^{*} \leq \mu^{*}$ and $\tilde{\nu}^{*} \leq
    \nu$ by lemma~\ref{cal}.  We have $\tilde{\mu}^{*} \preqsd \tilde{\nu}^{*}$ by
    \cite[proposition~3]{kamae1977stochastic}.  It follows that $(\tilde{\mu}^{*},
    \tilde{\nu}^{*}) \in \Phi(\mu^{*}, \nu)$.  We have $g(\mu^{*}, \nu) \leq 1 -
    \|\tilde{\mu}^{*}\| = \epsilon$.

    By a symmetric argument, we also have $g(\nu, \mu^{*}) \leq \epsilon$.  Hence
    $\gamma(\nu, \mu^{*}) \leq 2 \epsilon$.  Recalling the definition of $\nu$, 
    we have now shown that, $\A m \geq n_{\epsilon}$, $\gamma(\mu_{m}, \mu^{*}) \leq 2 \epsilon$.
    Since $\epsilon$ was arbitrary this concludes the proof.
\end{proof}

\subsection{Proofs of Section~\ref{s:applic} Results}

We begin with some lemmata.  

\begin{lemma}
    \label{l:ed}
    If $P$ is monotone, then
        $\sigma(P) = \inf_{(\mu, \nu) \in \pP \times \pP}  
                \alpha_O(\mu P, \nu P)$.
\end{lemma}

\begin{proof}
    Let $P$ be a monotone Markov kernel. It suffices to show that 
    the inequality 
    $\sigma(P) \leq \inf_{(\mu, \nu) \in \pP \times \pP}  
    \alpha_O(\mu P, \nu P)$ holds, since the reverse inequality is obvious.
    By the definition of $\sigma(P)$ and the identities in \eqref{eq:affco}, the
    claim will be established if we can show that
    \begin{equation}
        \label{eq:eled}
        \sup_{x,y} 
            \sup_{I \in i\bB} \{ P(x, I) - P(y, I) \}
        \geq
        \sup_{\mu,\nu} 
            \sup_{I \in i\bB} \{ \mu P(I) - \nu P(I) \}.
    \end{equation}
    where $x$ and $y$ are chosen from $S$ and $\mu$ and $\nu$ are chosen from
    $\pP$.  Let $s$ be the value of the right hand side of \eqref{eq:eled} and
    let $\epsilon > 0$.  Fix $\mu, \nu \in \pP$ and $I \in i\bB$ such that
    $\mu P(I) - \nu P(I) > s - \epsilon$, or, equivalently,
    \begin{equation*}
        \int \{ P(x, I) - P(y, I) \} (\mu \times \nu)(dx, dy) 
            > s - \epsilon.     
    \end{equation*}
    From this expression we see that there are $\bar x, \bar y \in S$ such
    that $ P(\bar x, I) - P(\bar y, I)  > s - \epsilon$.  Hence
        $\sup_{x,y} 
            \sup_{I \in i\bB} \{ P(x, I) - P(y, I) \} \geq s$,
    as was to be shown.
\end{proof}

\begin{lemma} 
    \label{l:cacher}
    If $\mu, \nu \in \mM$ and $(\tilde{\mu}, \tilde{\nu})$ is an ordered
    component pair of $(\mu, \nu)$, then
    \begin{equation*} 
        g(\mu P, \nu P) \leq g((\mu - \tilde{\mu}) P, (\nu - \tilde{\nu}) P).
    \end{equation*}
\end{lemma}

\begin{proof}
    Fix $\mu, \nu$ in $\mM$ and $(\tilde{\mu}, \tilde{\nu}) \in \Phi(\mu, \nu)$.
    Consider the residual measures
    $\hat{\mu} := \mu - \tilde{\mu} $ and $\hat{\nu} := \nu - \tilde{\nu}$.
    Let $(\mu', \nu')$ be a maximal ordered component pair of $(\hat{\mu} P,
    \hat{\nu} P)$, and define
    \begin{equation*}
        \mu^{*} := \mu' + \tilde{\mu} P 
        \quad  \text{and} \quad
        \nu^{*} := \nu' + \tilde{\nu} P.
    \end{equation*}
    We claim that
    $(\mu^{*}, \nu^{*})$ is an ordered component pair for $(\mu P, \nu P)$.
    To see this, note that
    \begin{equation*}
        \mu^{*} 
        = \mu' + \tilde{\mu} P 
        \leq \hat{\mu} P + \tilde{\mu} P 
        = (\hat{\mu} + \tilde{\mu}) P 
        = \mu P,
    \end{equation*}
    and, similarly, $\nu^{*} \leq \nu P$.  The measures $\mu^*$ and $\nu^*$ also have
    the same mass, since
    \begin{equation*} 
        \|\mu^{*}\| 
         = \|\mu'\| + \|\tilde{\mu} P\| 
         = \|\mu'\| + \|\tilde{\mu}\|
         = \|\nu'\| + \|\tilde{\nu}\| 
         = \|\nu'\| + \|\tilde{\nu} P\| 
         = \|\nu^{*}\|.
    \end{equation*}
    Moreover, since $\tilde{\mu} \preqsd
    \tilde{\nu}$ and $P$ is monotone, we have $\tilde{\mu} P \preqsd \tilde{\nu} P$.
    Hence $\mu^{*} \preqsd \nu^{*}$, 
    completing the claim that $(\mu^{*}, \nu^{*})$ is an ordered component
    pair for $(\mu P, \nu P)$.  As a result,
    \begin{align*}
        g(\mu P, \nu P) & \leq \|\mu P\| - \|\mu^{*}\| 
        = \|\mu P\| - \|\mu'\| - \|\tilde{\mu} P\| 
        \\
        & = \|\mu\| - \|\mu'\| - \|\tilde{\mu}\| = \|\hat{\mu}\| - \|\mu'\| 
        = \|\hat{\mu} P\| - \|\mu'\| = g(\hat{\mu} P, \hat{\nu} P).
        \qedhere
    \end{align*}
\end{proof}

\begin{proof}[Proof of theorem~\ref{t:ue}]
    Let $\mu, \nu \in \pP$.  Let $(\tilde{\mu}, \tilde{\nu})$ be a maximal
    ordered component pair for $(\mu, \nu)$.  Let $\hat{\mu} = \mu -
    \tilde{\mu}$ and $\hat{\nu} = \nu - \tilde{\nu}$ be the residuals.  Since
    $\|\tilde{\mu}\| = \|\tilde{\nu}\|$, we have $\|\hat{\mu}\| =
    \|\hat{\nu}\|$.  Suppose first that $\|\hat{\mu}\| > 0$.  Then $\hat{\mu}
    P/\|\hat{\mu}\|$ and $\hat{\nu} P/\|\hat{\mu}\|$ are both in $\pP$.  Thus, by
    lemma~\ref{l:ed},
    \begin{equation*}
        1 - \alpha_O(\hat{\mu} P /\|\hat{\mu}\|, \hat{\nu} P/\|\hat{\mu}\|) 
        \leq 1 - \sigma(P).
    \end{equation*}
    Applying the positive homogeneity property in lemma~\ref{l:poaf2} 
    yields
    \begin{equation*}
        \|\hat{\mu}\| - \alpha_O(\hat{\mu} P , \hat{\nu} P) 
        \leq (1 - \sigma(P)) \, \|\hat{\mu}\|,
    \end{equation*}
    Note that this inequality trivially holds if $\|\hat{\mu}\| = 0$.  
    From the definition of $g$, we can write the same inequality as
    $g(\hat{\mu} P, \hat{\nu} P) \leq (1 - \sigma(P)) g(\mu, \nu)$.
    If we apply lemma~\ref{l:cacher} to the latter we obtain
    \begin{equation*}
        g(\mu P, \nu P) \leq (1 - \sigma(P)) g(\mu, \nu).
    \end{equation*}
    Reversing the roles of $\mu$ and $\nu$, we also have $g(\nu P, \mu P) \leq
    (1 - \sigma(P)) g(\nu, \mu)$.  Thus
    \begin{multline*}
        \gamma(\mu P, \nu P) 
        = g(\mu P, \nu P) + g(\nu P, \mu P)
        \\
        \leq (1 - \sigma(P)) [g(\mu, \nu) + g(\nu, \mu)] 
        = (1 - \sigma(P)) \gamma(\mu, \nu),
    \end{multline*}
    verifying the claim in \eqref{eq:fbod}.  
\end{proof}

\begin{proof}[Proof of lemma~\ref{l:cbi}]
        To see that \eqref{eq:ev} holds, fix $\xi > \sigma(P)$ and suppose first that $\sigma(P)  = 1$.  Then 
    \eqref{eq:ev} holds because the right hand side of 
    \eqref{eq:ev} can be made strictly negative by choosing $\mu, \nu \in \pP$ 
    to be distinct.  Now suppose that $\sigma(P) < 1$ holds. It suffices to show that 
    \begin{equation}
        \label{eq:ev2}
        \forall \, \epsilon > 0, \;
        \exists \, x, y \in S 
        \text{ such that } 
        y \preceq x, \; x \npreceq y \text{ and } 
        \alpha_0(P_x, P_y) < \sigma(P) + \epsilon.
    \end{equation}
    Indeed, if we take \eqref{eq:ev2} as valid, set
    $\epsilon := \xi - \sigma(P)$ and choose $x$ and $y$ to satisfy the conditions in
    \eqref{eq:ev2}, then we have $\gamma(P_x, P_y) 
    = 2 - \alpha_0(P_x, P_y) - \alpha_0(P_y, P_x) > 2 - \xi  - 1 = 1 - \xi  = (1 - \xi)
    \gamma(\delta_x, \delta_y)$.  Therefore \eqref{eq:ev} holds.

    To show that \eqref{eq:ev2} holds, fix $\epsilon > 0$. We can use $\sigma(P) < 1$ and the definition of
    $\sigma(P)$ as an infimum to choose an $\delta \in (0, \epsilon)$ and points $\bar x, \bar y \in
    S$ such that $\alpha_0(P_{\bar x}, P_{\bar y}) < \sigma(P) + \delta < 1$.
    Note that $\bar x \preceq \bar y$ cannot hold here, because then 
    $\alpha_0(P_{\bar x}, P_{\bar y}) = 1$, a contradiction.  So suppose 
    instead that $\bar x \npreceq \bar y$. 
    Let $y$ be a lower bound of $\bar x$ and $\bar y$
    and let $x := \bar x$.  We claim that \eqref{eq:ev2} holds for the pair
    $(x,y)$.

    To see this, observe that, by the monotonicity result in lemma~\ref{l:mid}
    and $y \preceq \bar y$ we have
        $\alpha_0(P_x, P_y)
        = \alpha_0(P_{\bar x}, P_y)
        \leq \alpha_0(P_{\bar x}, P_{\bar y})
        < \sigma(P) + \delta
        < \sigma(P) + \epsilon$.
    Moreover, $y \preceq x$ because $x = \bar x$ and $y$ is by definition a lower
    bound of $\bar x$.  Finally, $x \npreceq y$ because if not then $\bar x =
    x \preceq y \preceq \bar y$, contradicting our assumption that $\bar x
    \npreceq \bar y$.
\end{proof}

\bibliographystyle{plain}

\bibliography{localbib}

\end{document}